\documentclass[5p]{elsarticle}

\usepackage{amsthm}

\usepackage{amsmath} 
\usepackage{amssymb} 
\usepackage{algorithm}
\usepackage[noend]{algpseudocode}
\algtext*{Until}
\usepackage{graphicx}
\usepackage[utf8]{inputenc}
\usepackage{todonotes}
\usepackage{hyperref}
\usepackage{braket}
\usepackage{pgfplots}
\usepackage{url}

\usepackage{subcaption}
\pgfplotsset{compat=1.11}

\newtheorem{theorem}{Theorem}
\newtheorem{definition}{Definition}
\newtheorem{lemma}{Lemma}

\newtheorem{remark}{Remark}
\newtheorem{assumption}{Assumption}

\DeclareMathOperator*{\argmin}{arg\,min}

\newcommand{\bbR}{\mathbb{R}}
\newcommand\pcap{\mathrel{\ooalign{\scalebox{1.1}{$\cap$}\cr
  \hidewidth\raise.25ex\hbox{\tiny$<\mkern2.2mu$}\cr}}}

\begin{document}
\begin{frontmatter}
\title{Prioritized Constraints in Optimization-Based Control}

\author[auth1]{Daniel Arnström}\ead{daniel.arnstrom@gmail.com}
\author[auth2]{Gianluca Garofalo}\ead{gianluca.garofalo@outlook.com}

\address[auth1]{Department of Information Technology, Uppsala University, Sweden}
\address[auth2]{Nio Robotics, Toulouse, France}

\begin{keyword}                           
    Optimization, Prioritized constraints, Predictive control
\end{keyword}

\begin{abstract}
    We provide theoretical foundations and computational tools for the systematic design of optimization-based control laws with constraints that have different priorities. By introducing the concept of \emph{prioritized intersections}, we extend and unify previous work on the topic. Moreover, to enable the use of prioritized intersection in real-time applications, we propose an efficient solver for forming such intersections for polyhedral constraints. The solver in question is a tailored implementation of a dual active-set quadratic programming solver that leverages the particular problem structure of the optimization problems arising for prioritized intersections.
The method is validated in a real-time MPC application for autonomous driving, where it successfully resolves six different levels of conflicting constraints, confirming its efficiency and practicality for control. Furthermore, we show that the proposed solver outperforms existing solvers for hierarchical quadratic programming, making it relevant beyond control applications. 
\end{abstract}

\end{frontmatter}

\definecolor{set19c1}{HTML}{E41A1C}
\definecolor{set19c2}{HTML}{377EB8}
\definecolor{set19c3}{HTML}{4DAF4A}
\definecolor{set19c4}{HTML}{984EA3}
\definecolor{set19c5}{HTML}{FF7F00}
\definecolor{set19c6}{HTML}{FFFF33}
\definecolor{set19c7}{HTML}{A65628}
\definecolor{set19c8}{HTML}{F781BF}
\definecolor{set19c9}{HTML}{999999}

\definecolor{prio3}{HTML}{E41A1C}
\definecolor{prio4}{HTML}{4DAF4A}
\definecolor{prio5}{HTML}{377EB8}

\newcommand{\makeplot}[3]{
    \begin{tikzpicture}[scale=1]
        \def\si{#1}
        \def\sii{#2}
        \def\siii{#3}
        \pgfplotstableread{data/scenario\si\sii\siii.dat}{\data}
        \begin{axis}[
            xmin=0,xmax=20,
            %ymode=log,
            ymin=-1.5,ymax=1.5,
            xlabel={Time [s]},
            ylabel={Position [m]},
            legend style={at ={(0.5,1.2)},anchor=north}, 
            ymajorgrids,yminorgrids,xmajorgrids,
            x post scale=0.6,
            y post scale=0.35,
            legend cell align={left},legend columns=3,
            ]
            \draw[fill=prio\si, opacity=0.5] (axis cs:5,-1.5) rectangle (axis cs:15,0.5);
            \draw[fill=prio\sii, opacity=0.5] (axis cs:7,-0.1) rectangle (axis cs:13,1.5);
            \draw[fill=prio\siii, opacity=0.5] (axis cs:9,-1.5) rectangle (axis cs:11,0.4);
            \addplot [black,very thick] table [x={t}, y={s}] {\data}; 
        \end{axis}
    \end{tikzpicture}
}

\section{Introduction}
\label{sec:intro}
Many real-world control problems involve multiple constraints that should be satisfied simultaneously, though some may take precedence over others. Consider, for example, a self-driving car scenario with the following safety constraints: (i) avoid collisions, (ii) stay on the road, and (iii) adhere to speed limits. Ideally, all of these constraints should be satisfied, but in case some of them are conflicting, (i) should be prioritized over (ii), and (ii) should be prioritized over (iii). 
More generally, control applications often have constraints with the following hierarchical structure: safety constraints, efficiency constraints, and durability constraints. This hierarchy can be seen as a generalization of Asimov’s famous \textit{Three Laws of Robotics} \cite{murphy2020beyond}.

A class of controllers that easily account for constraints is \textit{optimization-based} controllers, which produce control actions $u^*$ by solving optimization problems of the form 
    \begin{equation}
        \label{eq:opt-ctrl}
        u^* = \argmin_{u \in \mathcal{U}(x)} J(u,x),
    \end{equation}
where $x$ denotes the current state, $\mathcal{U}(x)$ denotes the set of admissible control actions at $x$, and $J$ denotes the control objective. Example of such controllers include model predictive controllers (MPCs) \cite{rawlings2017model}, safety filters \cite{wabersich2023data}, reference governors \cite{garone2017reference}, feasibility governors \cite{skibik2021feasibility}, and control allocators \cite{johansen2013control}. 

    Often in optimization-based controllers, all constraints that comprise the set $\mathcal{U}(x)$ are prioritized equally, even though some constraints are more important than others. Sometimes, a rough prioritization is imposed by classifying constraints as either \textit{hard} or \textit{soft}, where hard constraints are forced to hold, while soft constraints may be violated if necessary \cite{scokaert1999feasibility,zeilinger2014soft}. As highlighted by our initial example, there are, however, scenarios when finer-grained prioritization of constraints is necessary.
In this paper, we consider hierarchies of constraints with different importance. The usual soft/hard-dichotomy is a special case in this framework, which corresponds to a hierarchy with two levels.
More formally, we consider cases where $\mathcal{U}(x)$ is defined as the intersection of an ordered collections of feasibility sets, $\mathcal{U}(x) = \bigcap_{i=1}^p \mathcal{U}_i(x)$, 
where a lower index means higher priority. 
When the feasibility sets are compatible ($\mathcal{U}(x) \neq \emptyset$), the optimization-based controller can, without any intervention, directly use $\mathcal{U}(x)$ as the feasible set. If, on the other hand, the feasibility sets are incompatible ($\mathcal{U}(x) = \emptyset$,) we aim to ensure that as many high-prioritized constraints as possible are fulfilled.

Using optimization-based controllers of the form \eqref{eq:opt-ctrl} can be computationally demanding, since an optimization problem needs to be solved in real time. Including prioritized constraints in \eqref{eq:opt-ctrl} makes it even more demanding to solve, since conventional solvers cannot be applied. To address this, we propose an efficient solver for handling prioritized constraints in an important special case: when $J$ is quadratic and $\{\mathcal{U}_i\}_{i=1}^p$ are polyhedra.

%In this paper, we assume that a prioritization among the constraints is predefined and focus on developing a control framework that enforces it. Determining the particular prioritization is application dependent, and can lead to difficult ethical dilems reminiscent of the \textit{trolley problem} \cite{anderson2011machine}. Hence, the focus here is on how to construct control laws that actualize a given prioritization of constraints, rather than determining the prioritiztion itself. We focus on the problem of making control systems abide by ordered Asimov-like rules, rather than the construction such rules.

In summary, the contributions of this paper are:
\begin{enumerate}
    \item We introduce the notion of a prioritized intersection (Definition~\ref{def:pcap}), show how it can be computed, and derive several of its properties. This unifies and extends previous works. 
%    \item We extend classical invariance principle to prioritized intersections.
    \item We propose an efficient way of computing prioritized intersections for polyhedra (Algorithm~\ref{alg:main}), which is suitable for real-time applications.
    \item We use prioritized constraints in a real-time model predictive control application, and show that prioritized intersections can resolve conflicting constraints correctly and efficiently in real-time (Section~\ref{sec:mpc-exp}).
\end{enumerate}

\paragraph*{Related work}
Constraints of different priority naturally arise in robotics, where hierarchical quadratic programming \cite{escande2014hierarchical} is a principled way of handling prioritized \emph{linear} constraints. By using Newton steps, the same framework can be extended to handle nonlinear constraints \cite{pfeiffer2023hierarchical}. While state-of-the-art solvers such as \texttt{lexls} \cite{dimitrov2015efficient} and \texttt{NIPM} \cite{pfeiffer2023n} efficiently handle prioritized \emph{equality} constraints, they become inefficient when \emph{inequality} constraints also need to be prioritized. This restricts their use in real-time control applications, where inequality constraints are more common than equality constraints. In this paper, we propose a solver that handles prioritized equality \emph{and} inequality constraints efficiently.  

For autnonmous driving, prioritized constraints have been considered through so-called \emph{rulebooks} \cite{censi2019liability}. Safety filters that are based on control barrier functions have also recently been extended to safety constraints of different priorities \cite{basso2020task,lee2023hierarchical}.
In the context of MPC, the need for handling prioritized constraints is not new. In \cite{tyler1999propositional}, the authors use propositional logic to capture prioritized constraints. The strategy results in a mixed integer problem being solved. 
In \cite{kerrigan2002designing}, the authors use lexicographical optimization to handle prioritized objectives and constraints.
In \cite{vada2001linear}, prioritized constraints are handled by a two-step approach, where a linear program is first solved to get the ``optimal'' constraint slacks, followed by a solution of the nominal problem with the constraints relaxed using these slacks. The first step can be interpreted as forming the \textit{prioritized intersection} that we propose in this paper. 
There has also been work on the related problem of prioritized \emph{objectives} (rather than constraints) for MPC \cite{he2021lexicographic}.

All of these works do indiretctly make use of the prioritized intersection that we formalize in this paper.

\section{Prioritized intersections}
When multiple constraints are imposed simultaneously, feasible points are found in the intersection of the constraint sets. If the constraints have different importance, however, the standard commutative set intersection $\cap$ is insufficient, since it fails to distinguish between differently prioritized constraints \cite{fulton2013intersection}. To address this limitation, we introduce a noncommutative operator we call the \textit{prioritized intersection}, denoted $\pcap$, where the set $\mathcal{Z}_1 \pcap \mathcal{Z}_2$ represents the intersection of $\mathcal{Z}_1$ with $\mathcal{Z}_2$, with elements in $\mathcal{Z}_1$ being prioritized.

To formalize $\pcap$, we first associate each constraint set $\mathcal{Z}$ with a \textit{violation function} $\mathcal{V}_{\mathcal{Z}}$ that quantify how far away a point is from the set.
\begin{definition}[Violation function]
    \label{def:violation}
    A mapping $\mathcal{V}_{\mathcal{Z}}: \bbR^{n_z} \to \bbR$ is a \textit{violation function} for the set $\mathcal{Z}$ if 
    $\mathcal{V}_{\mathcal{Z}}(z) > 0$ for all $z \notin \mathcal{Z}$, and $\mathcal{V}_{\mathcal{Z}}(z) = 0$ for all $z \in \mathcal{Z}$. Moreover, for two points $\tilde{z}$,$\hat{z} \in \bbR^{n_z}$, we say that $\tilde{z}$ violates $\mathcal{Z}$ more than $\hat{z}$ if  $\mathcal{V}_{\mathcal{Z}}(\tilde{z}) >  \mathcal{V}_{\mathcal{Z}}(\hat{z})$.
\end{definition}

The prioritized intersection of two sets $\mathcal{Z}_1$ and $\mathcal{Z}_2$ is defined as the points in $\mathcal{Z}_1$ which violates $\mathcal{Z}_2$ the least, formalized as follows.
\begin{definition}[Prioritized intersection]
    \label{def:pcap}
Let $\mathcal{Z}_1,\mathcal{Z}_2 \subseteq \bbR^{n_z}$ and let $\mathcal{V}_{\mathcal{Z}_2}$ be a violation function for $\mathcal{Z}_2$. Then the \textit{prioritized intersection} of $\mathcal{Z}_1$ and $\mathcal{Z}_2$ is defined as
    \begin{equation}
        \label{eq:pcap}
        \mathcal{Z}_1 \pcap_{\mathcal{V}_{\mathcal{Z}_2}} \mathcal{Z}_2 \triangleq \argmin_{z \in \mathcal{Z}_1} \mathcal{V}_{\mathcal{Z}_2}(z). 
    \end{equation}
\end{definition}

Note that the prioritized intersection depends on which violation function is used. To simplify notation, we will drop the subscript and write $\pcap$ when it is clear/unimportant which violation function is used. 

To ensure that $\pcap$ defined above is suitable to handle prioritized constraints, we establish some of its properties in the following theorem.

\begin{theorem}[Properties of $\protect \pcap$]
    Let $\mathcal{Z}_1, \mathcal{Z}_2 \in \bbR^{n_z}$, and the operator $\pcap$ be as defined in Definition~\ref{def:pcap}. Then the following properties hold
    \begin{enumerate}
        \item $\mathcal{Z}_1 \pcap \mathcal{Z}_2 \subseteq \mathcal{Z}_1$. \label{prop:subset}
        \item $\mathcal{Z}_1 \pcap \mathcal{Z}_2 = \mathcal{Z}_1 \cap \mathcal{Z}_2 $ if  $\mathcal{Z}_1 \cap \mathcal{Z}_2 \neq \emptyset$ \label{prop:gencap}
        \item $\mathcal{Z}_1 \pcap \mathcal{Z}_2 \neq \emptyset$ iff $\mathcal{Z}_1 \neq \emptyset$ \label{prop:nonempty}.
    \end{enumerate}
\end{theorem}
\begin{proof}
    Property \ref{prop:subset} follows directly from $\pcap$ being defined by an optimization problem over the elements in $\mathcal{Z}_1$.
    For Property \ref{prop:gencap}, assume that $\mathcal{Z}_1 \cap \mathcal{Z}_2 \neq \emptyset$ and let $\tilde{z} \in \mathcal{Z}_1 \cap \mathcal{Z}_2$. Then, since $\tilde{z} \in \mathcal{Z}_2$, we have that $\mathcal{V}_{\mathcal{Z}_2}(\tilde{z})=0$. Moreover, the positive definitiveness of $\mathcal{V}_{\mathcal{Z}_2}$, and that $\tilde{z} \in \mathcal{Z}_1$, implies that $\tilde{z}$ is an optimizer to \eqref{eq:pcap}. Since $\tilde{z}$ is an arbitrary element in $\mathcal{Z}_1 \cap \mathcal{Z}_2$, we have that  
    \begin{equation*}
        \mathcal{Z}_1 \pcap \mathcal{Z}_2 \triangleq \argmin_{z \in \mathcal{Z}_1} \mathcal{V}_{\mathcal{Z}_2}(z) = \mathcal{Z}_1 \cap \mathcal{Z}_2, \quad\text{if } \mathcal{Z}_1 \cap \mathcal{Z}_2 \neq \emptyset.
    \end{equation*}
    Finally, for Property \ref{prop:nonempty}, $\mathcal{Z}_1 = \emptyset$ directly implies that there are no solutions to \eqref{eq:pcap}, and hence we get $\mathcal{Z}_1 \pcap \mathcal{Z}_2 = \emptyset$.
    If $\mathcal{Z}_1$ is nonempty, it follows from $\mathcal{V}_{\mathcal{Z}_2}$ being bounded from below by $0$ that the set of minimizers are nonempty, i.e., that $\mathcal{Z}_1 \pcap \mathcal{Z}_2 \neq \emptyset$.
\end{proof}

First, Property \ref{prop:subset} ensures that the prioritized intersection is a subset of the set with the highest priority, which yields the desired prioritization. Secondly, Property \ref{prop:gencap} ensures that the prioritized intersection recovers the normal intersection when both sets are compatible. Finally, Property \ref{prop:nonempty} ensures that the prioritized intersection always returns a nonempty set, even when the sets are conflicting, which highlights that $\pcap$ are able to resolve conflicting constraints.

\subsection{Explicit representation of prioritized intersections}
Next, we give a more explicit representation of $\pcap$ by considering the case when the constraint sets are sublevel sets of the form $\mathcal{Z}= \{z : g(z) \leq 0\}$ for some function $g: \bbR^{n_z} \to \bbR^m$.
In such cases, a natural measure of how much a point violates $\mathcal{Z}$ is
\begin{equation}
    \label{eq:distsub}
    \mathcal{V}_{\mathcal{Z}}(z)=\|\max(0,\Lambda g(z))\|_2^2, 
\end{equation}
where the max function is applied element-wise, and the weighting matrix $\Lambda$ is diagonal and positive definite.
When \eqref{eq:distsub} is used for prioritized intersections, the weighting matrix $\Lambda$ allows for a soft prioritization among the constraints that comprise $\mathcal{Z}$, where a large value on its diagonal corresponds to a higher prioritization of the corresponding constraint. 

The following lemma establishes that the function in \eqref{eq:distsub} is, in fact, a violation function. 
\begin{lemma}[Natural violation function]
    Let the set $\mathcal{Z} \subseteq \bbR^{n_z}$ be given by the sublevel set $\mathcal{Z} = \{z:g(z) \leq 0\}$. Then the function in \eqref{eq:distsub} is a violation function for $\mathcal{Z}$.
\end{lemma}
\begin{proof}
    If $z \notin \mathcal{Z}$, then we have $g(z) > 0$. resulting in $\max(0, \Lambda g(z)) > 0$, and, in turn, that $\mathcal{V}_{\mathcal{Z}}(z) >0$. If instead $z \in \mathcal{Z}$ we have $g(z) \leq 0$, which gives $\max(0,\Lambda g(z)) = 0$, and, in turn, $\mathcal{V}_{\mathcal{Z}}(z) = 0$. Hence, according to Definition~\ref{def:violation}, $\|\max(0,\Lambda g(z))\|_2^2 $ is a violation function.
\end{proof}
For the rest of the paper, we assume that the sets we want to intersect are sublevel sets so that we can leverage the natural violation function in \eqref{eq:distsub}.
\begin{assumption}[Sublevel sets]
   The sets to intersect are sublevel sets. 
\end{assumption}
\begin{assumption}[Natural violation function]
    The violation to a set $\mathcal{Z}= \{z: g(z) \leq 0\}$ is quantified with the natural violation function given in \eqref{eq:distsub}.
\end{assumption}
\begin{remark}[Alternative violation function] For sets that are not given by a sublevel set, a candidate for the violation functions is the projection distance to the set (since every distance to a set is also a violation function.)
\end{remark}

When using the natural violation function in \eqref{eq:distsub}, the prioritized intersection $\pcap$ is explicitly given as the intersection of the higher-prioritized set and a perturbation of the sublevel set that defines the lower prioritized set, as is shown in the following lemma.

\begin{lemma}[Explicit representation of $\pcap$] 
    \label{lem:exp}
    Let $\mathcal{Z}_1 \subseteq \bbR^{n_z}$ and let $\mathcal{Z}_2 = \{z\in \bbR^{n_z} : g(x) \leq 0\}$. Moreover, let the violations to $\mathcal{Z}_2$ be quantified by the natural violation function in \eqref{eq:distsub}. Then there exists $\epsilon^* \in \bbR^m_{\geq 0}$ such that
    \begin{equation*}
        \mathcal{Z}_1 \pcap \mathcal{Z}_2 = \mathcal{Z}_1 \cap \{z : g(z) \leq \epsilon^* \}.
    \end{equation*}
\end{lemma}
\begin{proof} First note that $\max(0,c)$, for any $c\in \mathbb{R}$, is equivalent to  $\max(0,c) = \min_{\tilde{\epsilon}\geq 0} \tilde{\epsilon}$ subject to $c \leq \tilde{\epsilon}$. Thus 
        $\|\max(0,\Lambda g(z))\|^2_2 = \min_{\tilde{\epsilon}} \|\tilde{\epsilon}\|^2_2 \text{ subject to } \Lambda g(z) \leq \tilde{\epsilon}.$
    The variable change $\epsilon = \Lambda^{-1} \tilde{\epsilon}$ gives
    \begin{equation*}
        \|\max(0,\Lambda g(z))\|^2_2 = \min_{\epsilon} \|\Lambda \epsilon \|^2_2 \text{ subject to }  g(z) \leq \epsilon.
    \end{equation*}
    Inserting this into \eqref{eq:pcap} and combining the minimization over $\epsilon$ and $z$ yields that $\mathcal{Z}_1 \pcap \mathcal{Z}_2$ is the optimizers to
    \begin{equation}
        \label{eq:epipcap}
            \min_{z\in \mathcal{Z}_1, \epsilon} \|\Lambda \epsilon\|^2_2 \text{ subject to }  g(z) \leq \epsilon,
    \end{equation}
    where the minimizing $\epsilon$ is denoted $\epsilon^*$. Since $z$ only enters in the constraints of \eqref{eq:epipcap}, we get that all $z\in \mathcal{Z}_1$ satisfying $g(z) \leq \epsilon^*$ are minimizers.
\end{proof}

\begin{remark}[Geometry perservation]
    Lemma~\ref{lem:exp} shows that $\pcap$ preserves the geometry of $\mathcal{Z}_2$ by perturbing the right-hand-side of the sublevel set $g(x) \leq 0$ before intersecting it with $\mathcal{Z}_1$.
\end{remark}

\begin{remark}[Regularization]
    \label{rem:reg}
    The objective of the optimization problem in \eqref{eq:epipcap} is convex, but not strictly convex (since $z$ does not enter in the objective.) As a result, the problem does not necessarily have a unique solution, which can lead to irregularities when used for control. A remedy is to add a strictly convex term $q(z)$ to the objective, which yields a unique $\epsilon^*$, since the objective becomes strictly convex jointly in both $z$ and $\epsilon$. If, for example, $q(z) = \rho \|z\|^2_2$, we get for a sufficiently small $\rho>0$ that $\epsilon^*$ is an optimizer to \eqref{eq:epipcap} \cite{kerrigan2000soft}.
\end{remark}

\begin{remark}[Double-sided constraints]
    If $\mathcal{Z}$ is expressed by double-sided inequalities $\mathcal{Z} = \{z: \underline{b} \leq g(z) \leq \overline{b} \}$, the violation function $\mathcal{V}_{\mathcal{Z}}$ can be extended as  
    \begin{equation}
        \mathcal{V}_{\mathcal{Z}}(z)=\|\max(0,\Lambda (g(z)-\overline{b}), \Lambda(\underline{b}-g(z)))\|_2^2, 
    \end{equation}
    and the perturbation in \eqref{eq:epipcap} becomes 
    \begin{equation}
        \label{eq:epipcap-double}
        \min_{z\in \mathcal{Z}_1, \epsilon} \|\Lambda \epsilon\|^2_2 \text{ subject to }  \underline{b} - \epsilon \leq g(z) \leq \overline{b} + \epsilon.
    \end{equation}
\end{remark}
\subsection{Prioritized intersection of multiple sets}
Often, we want to intersect more than just two sets. We therefore introduce notation for iteratively applying $\pcap$ to an ordered collection of sets $\{\mathcal{Z}_i\}_{i=1}^p$. The prioritized intersection for such an ordered collection is denoted 
\begin{equation}
    \overset{p}{\underset{i=1}{\pcap}} \mathcal{Z}_i  \triangleq   \left(\overset{p-1}{\underset{i=1}{\pcap}} \mathcal{Z}_i\right) \pcap \mathcal{Z}_p, \text{   with } \overset{1}{\underset{i=1}{\pcap}} \mathcal{Z}_i = \mathcal{Z}_1.
\end{equation}
Examples of prioritized intersections are shown in Figure~\ref{fig:pcap}, where three different sets are intersected using $\pcap$. Note that the ordering is essential, and different orderings typically lead to different intersections. 
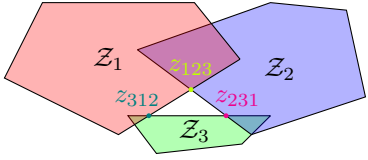
\begin{figure}
  \centering
  \begin{tikzpicture}[scale=0.5]
    \coordinate(p31) at (1,-1);
    \coordinate(p32) at (-2.75,-1);
    \coordinate(p33) at (-2,-2);
    \coordinate(p34) at (0.25,-1.75);
    \draw[fill=green, fill opacity = 0.3] (p31) -- (p32) -- (p33) -- (p34) -- (p31);
    \node (u3) at (-1,-1.4) {$\mathcal{Z}_3$};

    \coordinate(p21) at (-2.5,0.75);
    \coordinate(p22) at (1,2);
    \coordinate(p23) at (3,1.8);
    \coordinate(p24) at (3.5,-0.5);
    \coordinate(p25) at (0.5,-1.5);
    \draw[fill=blue , fill opacity = 0.3] (p21) -- (p22) -- (p23) -- (p24) -- (p25) -- (p21);
    \node (u2) at (1.25,0.25) {$\mathcal{Z}_2$};

    \coordinate(p11) at (0.2,0.5);
    \coordinate(p12) at (-1,2);
    \coordinate(p13) at (-5,2);
    \coordinate(p14) at (-6,0);
    \coordinate(p15) at (-3,-1.5);
    \draw[fill=red, fill opacity =0.3] (p11) -- (p12) -- (p13) -- (p14) -- (p15) -- (p11);
    \node (u1) at (-3.25,0.5) {$\mathcal{Z}_1$};

    \coordinate (optv1) at (intersection of p11--p15 and p21--p25){};
    \coordinate (optv2) at (intersection of p21--p25 and p31--p32){};
    \coordinate (optv3) at (intersection of p11--p15 and p31--p32){};

    \draw[fill,lime] (optv1) circle [radius=.05] node[above,yshift=0.5ex]{$z_{123}$};
    \draw[fill,magenta] (optv2) circle [radius=.05] node[above,xshift=1ex]{$z_{231}$};
    \draw[fill,teal] (optv3) circle [radius=.05] node[above,xshift=-1ex]{$z_{312}$};
\end{tikzpicture}
  \caption{Example of prioritized intersection, where  $\left(\mathcal{Z}_i \protect \pcap \mathcal{Z}_j\right) \protect \pcap \mathcal{Z}_k = \{z_{ijk}\}$.}
  \label{fig:pcap}
\end{figure}

To give a more explicit representation of $\overset{p}{\underset{i=1}{\pcap}} \mathcal{Z}_i$ , iteratively applying Lemma~\ref{lem:exp} gives that the prioritized intersection of several sets is just the intersection of perturbed versions of the nominal sublevel sets, as is shown in the following lemma.

\begin{lemma} 
    \label{lem:multi-pcap}
    Consider the ordered collection $\{\mathcal{Z}_j\}_{j=1}^p$ with $\mathcal{Z}_j=\{z : g_j(z) \leq 0\}$. Moreover, let $z^*_i$ and $\epsilon^*_i$ be defined by the hierarchy of optimization problems 
    \begin{equation}
        \label{eq:hi-opt}
        \begin{aligned}
            (z^*_i, \epsilon^*_i) \in \argmin_{z\in \mathcal{Z}_1,\epsilon_i} &\|\Lambda_i \epsilon_i\|_2^2  \\
            \text{subject to } &g_j(z) \leq \epsilon^*_j, \quad j = 2\dots,i-1.\\
                               & g_i(z) \leq \epsilon_i 
        \end{aligned}
    \end{equation}
Then $\overset{p}{\underset{i=1}{\pcap}}\mathcal{Z}_i = \mathcal{Z}_1 \cap \{z : g_i(z) \leq \epsilon^*_i,\quad i=2,\dots,p\}.$
\end{lemma}
\begin{proof}
    Follows by iteratively applying Lemma~\ref{lem:exp}.
\end{proof}

\subsection{Prioritized intersection of polyhedra}
Often in applications, especially in optimization-based controllers, the feasibility sets are polyhedra \cite{houska2024polyhedral} (i.e., the function $g$ defining the sublevel set is affine).
An important property is that polyhedra are closed under $\pcap$. 

\begin{lemma}[Polyhedra closed under $\pcap$]
    \label{lem:pcap-poly}
    If $\mathcal{Z}_1$ and $\mathcal{Z}_2$ are polyhedra, then $\mathcal{Z}_1 \pcap \mathcal{Z}_2$ is also a polyhedron.
\end{lemma}
\begin{proof}
    If $\mathcal{Z}_2$ is a polyhedron, it can be expressed as $\{z : Az \leq b\}$ for some matrix $A$ and vector $b$. Using Lemma~\ref{lem:exp} with $g(z) = Az - b$ then gives  
    \begin{equation}
        \label{eq:cappoly}
    \mathcal{Z}_1 \pcap \mathcal{Z}_2 = \mathcal{Z}_1 \cap \{z : A z \leq b + \epsilon^*\} 
    \end{equation}
    for some $\epsilon^*$. Since the right-hand side of \eqref{eq:cappoly} is an intersection of two polyhedra, and that polyhedra are closed under intersection, $\mathcal{Z}_1 \pcap \mathcal{Z}_2$ is also a polyhedron. 
\end{proof}

\section{Computing prioritized intersections}
Lemma~\ref{lem:multi-pcap} gives a direct way of computing the prioritized intersection by solving a sequence of optimization problems of the form \eqref{eq:hi-opt}. Naively solving these sequences can, however, hinder the use of $\pcap$ in real-time applications. To this end, we propose an efficient way of computing $\pcap$ when the sets are polyhedra. 
That is, we compute $\pcap$ for the sets $\{\mathcal{Z}_i\}_{i=1}^p$ with $\mathcal{Z}_j=\{z : A_j z - b_j \leq 0\}$ for some matrices $A_j \in \mathbb{R}^{m_j \times n_z}$ and vectors $b_j \in \mathbb{R}^{m_j}$. For polyhedra, \eqref{eq:hi-opt} becomes a sequence of quadratic programs. 

Concretely, we use Lemma~\ref{lem:multi-pcap} to compute $\pcap$ by determining the perturbations $\epsilon^*_i$.
For numerical reasons (see Remark~\ref{rem:reg}), we regularize the objective by adding the term $q(z) = \rho^2 \|z\|_2^2$, with the regularization parameter $\rho >0$. The resulting sequence of quadratic programs are 
\begin{equation}
    \label{eq:hi-qp}
    \begin{aligned}
        (z^*_i, \epsilon^*_i) = \argmin_{z,\epsilon_i} &\|\Lambda_i \epsilon_i\|_2^2 + \rho^2 \|z\|_2^2  \\
        \text{subject to } %& A_1 z \leq b_1, \\ 
                           &A_j z \leq b_j + \epsilon^*_j, \quad j = 1\dots,i-1.\\
                           &A_i z \leq b_i  + \epsilon_i.
    \end{aligned}
\end{equation}
from $i=2$ to $i=p$, with $\epsilon^*_1 \triangleq 0$.
Note that the regularization term $\rho^2 \|z\|_2^2$ leads to unique solutions in accordance with Remark~\ref{rem:reg}.

\subsection{Efficient computation of $\protect \pcap$ of polyhedra}
Efficiently solving the sequence of QPs in \eqref{eq:hi-qp} has been considered before in the context of hierarchical quadratic programming \cite{escande2014hierarchical} in robotics. Existing solvers for hierarchical quadratic programming include the primal active-set solver \texttt{lexls} proposed in \cite{dimitrov2015efficient}, and the interior-point solver \texttt{NIPM} proposed in \cite{pfeiffer2023n}. As mentioned in the introduction, these solvers can, however, be inefficient when there are inequality constraints, since they have mainly been developed for resolving prioritized \emph{equality} constraints. For optimization-based controllers, however, \emph{inequality} constraints needs to be handled efficiently. We, hence, propose an alternative method that is based on the dual active-set solver \texttt{DAQP} \cite{arnstrom2022daqp}. This solver efficiently solves quadratic programs that arise in embedded control applications, which we here extend to be applicable to hierarchies of the form \eqref{eq:hi-qp}. Two reasons why the active-set solver \texttt{DAQP} is more efficient than \texttt{lexls} (which is also an active-set solver) for inequality constraints is that it (i) exploits low-rank updates to intermediate matrix factorization when searching among inequality constraints, and (ii) operates on a dual problem, which typically results in fewer iterations \cite{goldfarbidnani}.

To be able to understand how \texttt{DAQP} can be tailored to efficiently compute prioritized intersections, we first briefly summarize how it solves problems of the form \eqref{eq:hi-qp} for a fixed level $i$; for a complete description, see \cite{arnstrom2022daqp}. Since \texttt{DAQP} is an \emph{active-set} solver, it tries to identify the so-called optimal active set, which is the inequality constraints that hold with equality at the optimum. If this set would be known, \eqref{eq:hi-qp} could easily be solved by solving a single system of linear equations. To find the optimal active set, \texttt{DAQP} updates a so-called \emph{working set} $\mathcal{W}\subseteq \mathbb{N}_{1:m}$, which contains indices of inequality constraints that are imposed to hold with equality. The working set $\mathcal{W}$ is updated by adding/removing constraints to/from it until it equals the optimal active set. To determine which constraint should be added/removed to/from $\mathcal{W}$, a system of linear equations (specifically, a \emph{KKT system}, see \eqref{eq:kkt} below for details) is solved. A high-level description of an iteration in \texttt{DAQP} is given in Algorithm~\ref{alg:as-methods-pseudo}, where $\lambda$ denotes dual variables for the constraints of \eqref{eq:hi-qp} (again, see \cite[Sec. II]{arnstrom2022daqp} for a more detailed description of \texttt{DAQP}.)
%In Section~\ref{sssec:imp-slack}, we show the particular structure of these equations for \eqref{eq:hi-qp}, and propose how the structure can be exploited to improve computational efficiency.
 
\begin{algorithm}[H]
    \caption{Prototypical active-set algorithm for \eqref{eq:hi-qp}.}
  \label{alg:as-methods-pseudo}
  \begin{algorithmic}[1]
      \Require Initial working set $\mathcal{W}$
      \Ensure Optimal primal/dual solution and active set 
	\Repeat
    \State $(\left[\begin{smallmatrix} z \\ \epsilon_i  \end{smallmatrix}\right], \lambda) \leftarrow$ Solve KKT system defined by $\mathcal{W}$ 
	\If{$(\left[\begin{smallmatrix} z \\ \epsilon_i  \end{smallmatrix}\right], \lambda)$ is primal and dual feasible} 
 \State \textbf{return} $(\left[\begin{smallmatrix} z^* \\ \epsilon^*_i  \end{smallmatrix}\right], \lambda^*,\mathcal{W}^*) \leftarrow (\left[\begin{smallmatrix} z \\ \epsilon^*_i  \end{smallmatrix}\right], \lambda,\mathcal{W})$
	%\State \textbf{return} $(x^*, \lambda^*, \mathcal{A}^*) \leftarrow (x, \lambda, \mathcal{W})$
	\Else
    \State Modify $\mathcal{W}$ based on and $\left[\begin{smallmatrix} z \\ \epsilon_i  \end{smallmatrix}\right]$ and $\lambda$.
	\EndIf
    \Until
    %\EndRepeat
  \end{algorithmic}
\end{algorithm}

We will now highlight two features that make \texttt{DAQP} particularly suitable for solving hierarchies of the form \eqref{eq:hi-qp}.

\subsubsection{Implicitly handling slack variables $\epsilon$}
\label{sssec:imp-slack}
To efficiently solve the hierarchy of QPs in \eqref{eq:hi-qp}, each QP needs to be solved efficiently. To this end, \texttt{DAQP} can reduce the number of decision variables by handling the slack variables $\epsilon$ implicitly. Exactly how this can be done requires more details of the internals of \texttt{DAQP} and of \eqref{eq:hi-qp}. 
With dual variables $\lambda$, the KKT conditions of \eqref{eq:hi-qp} (after divding the objective with $\rho^2$) are
\begin{subequations}
    \label{eq:kkt}
  \begin{align}
        &\begin{bmatrix}
            z \\ \frac{1}{\rho^2}\Lambda^2_i \epsilon_i
        \end{bmatrix}
         + 
         \begin{bmatrix}
             A_1^T & \cdots & A_{i-1}^T & A_i^T \\
             0 & \cdots & 0 & -I 
         \end{bmatrix}
         %\begin{bmatrix}
         %    \lambda_1 \\ \vdots \\ \lambda_{h-1} \\ \lambda_h
         %\end{bmatrix}
         \left[\begin{smallmatrix}
             \lambda_1 \\ \vdots \\ \lambda_i
     \end{smallmatrix}\right]
           = 
           0, \label{eq:kkt-stat}\\
        &0 \leq \lambda_j \perp b_j + \epsilon^*_j - A_j z \geq 0, \: j = 1,\dots, i-1, \label{eq:kkt-comp1} \\
        &0 \leq \lambda_i \perp b_i + \epsilon - A_i z \geq 0 \label{eq:kkt-comp2}.
  \end{align}
\end{subequations}
From the stationarity condition \eqref{eq:kkt-stat} we get that $\epsilon_i$ and $z$ are directly given by the dual variables $\lambda$; concretely
% , $\epsilon_i = \rho^2\Lambda_i^{-2} \lambda_i$ and $z = -A^T  \lambda$,
\begin{equation}
    \begin{aligned}
        &\epsilon_i = \rho^2\Lambda_i^{-2} \lambda_i, \\
        & z = -A^T  \lambda,
    \end{aligned}
\end{equation}
where $A^T \triangleq \begin{bmatrix}
    A^T_1 \cdots A^T_i 
\end{bmatrix}$ and $\lambda^T \triangleq 
\begin{bmatrix}
    \lambda_1^T \cdots \lambda_i^T 
\end{bmatrix}$. Note specifically that since $\Lambda_i$ is a diagonal matrix, $\epsilon_i$ is just a scaled version of $\lambda_i$. 

The structure of \eqref{eq:hi-qp} can be exploited further. First, let $\tilde{\epsilon}_i \triangleq \frac{1}{\rho}\Lambda_i \epsilon$. Then by defining 
\begin{equation}
    \label{eq:Md}
    M \triangleq 
\left[\begin{smallmatrix}
        A_1 & 0 \\
      \vdots & \vdots \\
      A_{i-1} & 0 \\
      A_{i} & -\rho\Lambda^{-1}
\end{smallmatrix}\right], \qquad 
  d\triangleq 
  \left[\begin{smallmatrix}
   b_1 + \epsilon_1 \\
   \vdots \\
   b_{i-1} + \epsilon_{i-1} \\
   b_i
\end{smallmatrix}\right],
\end{equation}
the stationarity condition \eqref{eq:kkt-stat} can be compactly written as
        $\left[\begin{smallmatrix}
                {z} \\ \tilde{\epsilon}_i
 \end{smallmatrix}\right]
         + M^T \lambda  = 0$,
and \eqref{eq:kkt-comp1}-\eqref{eq:kkt-comp2} can be compactly written as 
  $0 \leq \lambda \perp d- M 
  \left[\begin{smallmatrix}
          {z} \\ \tilde{\epsilon}_i
 \end{smallmatrix}\right] \geq 0$.

 For a given $\mathcal{W}$, we then have that $[M]_{\mathcal{W}} 
        \left[\begin{smallmatrix}
                z \\ \tilde{\epsilon}_i
 \end{smallmatrix}\right] = [d]_{\mathcal{W}}$ and $[\lambda]_i = 0$ for $i\notin \mathcal{W}$, where $[\cdot]_{\mathcal{W}}$ extracts rows of $M$ indexed by the working set $\mathcal{W}$ .
 This inserted into the stationarity condition gives that the dual variables are given by 
\begin{equation}
    \label{eq:daqp-linsolve}
    [M]_{\mathcal{W}} [M]_{\mathcal{W}}^T [\lambda]_{\mathcal{W}} =-[d]_{\mathcal{W}}.
\end{equation}
To solve these systems of linear equations efficiently, \texttt{DAQP} maintains an $LDL^T$ factorization with a lower triangular matrix $L$ and a diagonal matrix $D$ such that $[M]_{\mathcal{W}} [M]_{\mathcal{W}}^T = L D L^T$. The factors $L$ and $D$ are updated when the working set $\mathcal{W}$ changes.  

\begin{remark}[Lack of low-rank updates in \texttt{lexls}]
The solver \texttt{lexls} uses a lexicographic QR decomposition to efficiently handle equality constraints \cite{dimitrov2015efficient}. While well-known low-rank updates exists for the conventional QR decomposition, no such updates have yet been proposed for the \emph{lexicographical} QR decomposition (and it is unclear if such updates are possible.) 
%In \cite{escande2014hierarchical}, a hierarchical complete orthogonal decomposition (HCOD) is used to solve problems of the form \eqref{eq:hi-opt}. Even though low-rank updates to the HCOD are derived in \cite{escande2014hierarchical}, it is shown in \cite{dimitrov2015efficient} that the lexicographical QR decomposition is more efficient than the HCOD, even without low-rank updates. 
\end{remark}

Next we will show how one can use the lower dimensional $A$ instead of the full matrix $M$ to update the $LDL^T$ factorization, which is computationally benificial. First, we recall the following lemma (Theorem 2 in \cite{bemporad2016qp}) that is used in \texttt{DAQP} to recursively update an $LDL^T$ factorization.

\begin{lemma}
    \label{lem:LDLup-original}
    Let $L$ be a unit lower triangular matrix and $D$ be a diagonal matrix such that $[M]_{\mathcal{W}} [M]_{\mathcal{W}}^T = L D L^T$. Furthermore let $\tilde{M} = [\begin{smallmatrix} 
        [M]_{\mathcal{W}} \\ [M]_i 
    \end{smallmatrix}]$. Then $\tilde{M} \tilde{M}^T = \tilde{L} \tilde{D} \tilde{L}^T$ with 
      $\tilde{L} = 
  \begin{bmatrix}
	L &0\\
	l^T & 1
  \end{bmatrix}$ and 
  $\tilde{D} =
  \begin{bmatrix}
	D &0 \\
	0 & \delta
  \end{bmatrix},$
  where $l$ and $\delta$ are defined by
      $L D l = [M]_{\mathcal{W}} [M]_i^T$, $\delta = [M]_i [M]_i^T - l^T D l.$
\end{lemma}

The following novel result shows that only $A$ needs to be used when updating an $LDL^T$ factorization. As a result, matrix operations are carried out on matrices with $n_z$ columns instead of on matrices with $n_z + m_i$ columns.
\begin{lemma}
    \label{lem:LDLup1}
    Assume that $i \notin \mathcal{W}$ and that $M$ is defined as in \eqref{eq:Md}. Then $[M]_{\mathcal{W}} [M]_i^T = [A]_{\mathcal{W}} [A]_i^T$.
\end{lemma}
\begin{proof}
    Let the right block of $M$ in \eqref{eq:Md} be denoted $\mathcal{I}$; that is $\mathcal{I} \triangleq \left[\begin{smallmatrix}0 \\ -\rho \Lambda^{-1} \end{smallmatrix}\right]$, resulting in $M = [A \:\: \mathcal{I}]$.
We then get
%\begin{equation*}
%    \begin{aligned}
%        [M]_{\mathcal{W}} [M]_i^T &= 
%    \left[ [A]_{\mathcal{W}} \quad [\mathcal{I}]_{\mathcal{W}} \right]
%    \begin{bmatrix}
%        [A]_i^T \\ 
%        [\mathcal{I}]_i^T
%    \end{bmatrix}  \\
%    & = [A]_{\mathcal{W}} [A]_i^T + [\mathcal{I}]_{\mathcal{W}} [\mathcal{I}]_i^T.
%    \end{aligned}
%\end{equation*}
$[M]_{\mathcal{W}} [M]_i^T  = [A]_{\mathcal{W}} [A]_i^T + [\mathcal{I}]_{\mathcal{W}} [\mathcal{I}]_i^T.$
The structure of $\mathcal{I}$ in combination with $i \notin \mathcal{W}$ gives $[\mathcal{I}]_{\mathcal{W}} [\mathcal{I}]_i^T = 0$.
\end{proof}
\begin{lemma}
    \label{lem:LDLup2}
    Let $M$ be defined as in \eqref{eq:Md} and let $k\triangleq i - \sum_{j=1}^{h-1} m_j$. Then 
    \begin{equation*}
        [M]_{i} [M]_i^T = \begin{cases}
            [A]_{i} [A]_i^T + \rho^2 [\Lambda^{-2}]_{kk}, &\text{if } k \geq 0 \\
            [A]_{i} [A]_i^T, &\text{otherwise.}
        \end{cases}
    \end{equation*}
\end{lemma}
%\vspace{-20pt}
\begin{proof}
    Analogous to the proof of Lemma~\ref{lem:LDLup1}, we get that 
    $[M]_i [M]_i^T = [A]_{i} [A]_i^T + [\mathcal{I}]_{i} [\mathcal{I}]_i^T,$
    where $\mathcal{I} \triangleq \left[\begin{smallmatrix}0 \\ -\rho \Lambda^{-1} \end{smallmatrix}\right]$, with the zero block having $\sum_{j=1}^{h-1} m_j$ rows. Based on this, we get that 
$[\mathcal{I}]_{i} [\mathcal{I}]_i^T = 
\begin{cases}
    \rho^2 [\Lambda^{-2}]_{kk} , &\text{if } i \geq \sum_{j=1}^{h-1} m_j \\
            0. &\text{otherwise.}
        \end{cases}$
\end{proof}
An equivalent way of expressing the condition that $i \geq \sum_{j=1}^{h-1} m_j$ is that the $i$th row of $A$ belongs to $A_i$.

Now, Lemma~\ref{lem:LDLup-original} can be modified to only perform computations on $A$ instead of $M$, which reduces the number of numerical operations. 

\begin{theorem}
    Let $M$ be the matrix defined in \eqref{eq:Md}, $\mathcal{W}$ be an index set, and $i$ be an integer such that $i \notin \mathcal{W}$. Then $l$ and $\delta$ in Lemma~\ref{lem:LDLup-original} simplifies to 
%    Moreover, $L$ be a unit lower triangular matrix and $D$ be a diagonal matrix such that $[M]_{\mathcal{W}} [M]_{\mathcal{W}}^T = L D L^T$. Finally, let $\tilde{M} = [\begin{smallmatrix} 
%        [M]_{\mathcal{W}} \\ [M]_i 
%    \end{smallmatrix}]$. Then $\tilde{M} \tilde{M}^T = \tilde{L} \tilde{D} \tilde{L}^T$ with 
%  \begin{equation}
%      \tilde{L} = 
%  \begin{bmatrix}
%	L &0\\
%	l^T & 1
%  \end{bmatrix}, \quad
%  \tilde{D} =
%  \begin{bmatrix}
%	D &0 \\
%	0 & \delta
%  \end{bmatrix}, \\
%  \end{equation}
%  where $l$ and $\delta$ are defined by
  \begin{equation*}
      \begin{aligned}
          &L D l = [A]_{\mathcal{W}} [A]_i^T, \\ 
          &\delta = 
      \begin{cases}
          \rho^2[\Lambda^{-2}]_{kk}+[A]_i [A]_i^T - l^T D l &\text{if } i \geq \sum_{j=1}^{h-1} m_j \\
          [A]_i [A]_i^T - l^T D l &\text{otherwise.}
      \end{cases}
      \end{aligned}
  \end{equation*}
  %where $k\triangleq i - \sum_{j=1}^{h-1} m_j$.
\end{theorem}
\begin{proof}
    Follows directly by combining the results in Lemma~\ref{lem:LDLup-original}-\ref{lem:LDLup2}.
\end{proof}

\subsubsection{Warm-starting}
The number of KKT systems that needs to be solved can be reduced significantly if the initial working set $\mathcal{W}_0$ is close to the optimal active set, since this typically requires fewer constraint to be added/removed to/from $\mathcal{W}$. We, hence, warm-start \texttt{DAQP} to efficiently solve the sequence of problems in \eqref{eq:hi-qp}, summarized in Algorithm~\ref{alg:main}. 

\begin{algorithm}[H]
    \caption{Warm starting \texttt{DAQP} to solve \eqref{eq:hi-qp}.}
    \label{alg:main}
  \begin{algorithmic}[1]
      \Require Polyhedra $\{(A_i,b_i)\}_{i=1}^p$, weights $\{\Lambda_i\}^p_{i=2}$, \newline regularization $\rho >0$, initial working set $\mathcal{W}_0$. 
    \Ensure Perturbations $\{\epsilon^*_i\}_{i=2}^p$
    \For{$i \in \{2,\dots,p\}$}
    \State $(\epsilon^*_i, \mathcal{W}_i) \leftarrow$ solve \eqref{eq:hi-qp} using \texttt{DAQP} with $\mathcal{W}_{i-1}$. 
    \EndFor
  \end{algorithmic}
\end{algorithm}

%\begin{remark}[Hot start]
%    Even more advantageous than being warm-started, \texttt{DAQP} can be \textit{hot}-started between each iteration, by reusing its internal $LDL^T$ factorization.
%\end{remark}

\section{Numerical Experiments}
\label{sec:experiment}
In this section, we investigate how Algorithm~\ref{alg:main} compares to existing state-of-the-art solvers for solving hierarchical optimization problems of the form \eqref{eq:hi-qp}. We then show how prioritized intersections can be used in a real-time control application by using them in an autonomous driving scenario where constraints of different importance need to be imposed.

With the experiments\footnote{Code for all experiments is available at \url{https://github.com/darnstrom/hdaqp-experiments}} we aim to show that:
(i) Algorithm~\ref{alg:main} outperforms state-of-the-art solvers for computing prioritized intersections of polyhedra. 
(ii) Prioritized intersections are useful in control applications with prioritized constraints.
(iii) Prioritized intersections can be used in real-time applications. 

\subsection{Computation of $\protect \pcap$ for polyhedra}

In the experiments, we compute the prioritized intersection for $p=10$ sets of the form $\{z : \underline{b}_i \leq A_i z \leq \overline{b}_i\}$ with the elements of $A_i \in \mathbb{R}^{m_i \times n_z}$ drawn from a uniform distribution over $[0,1]$. The dimension of $z$ is $n_z = 50$, and the number of constraints $m_i$ in the $i$th set is drawn from a uniform distribution over $\{1,\dots,20\}$. The elements of the upper offset $\overline{b}_i$ are drawn from a uniform distribution over $[0,1]$, and are then perturb with a term $A_i \tilde{z}_i$, where elements of $\tilde{z}_i \in \bbR^{n_z}$ are drawn from a uniform distribution over $[-1,1]$. The perturbation results in $\tilde{z}_i$ being feasible for the $i$th set, ensuring a nonempty set. The lower bound $\underline{b}$ is set equal to $\overline{b}$ with a probability of $q$, and otherwise it is set to $\overline{b}_i-v$, where the elements of $v$ are drawn from a uniform distribution over $[0,1]$ (which ensures that $\underline{b}_i \leq \overline{b}_i$.) We solve the resulting hierarchy of optimization problems of the form \eqref{eq:hi-qp} using Algorithm~\ref{alg:main} (\texttt{DAQP}) and the state-of-the art hierarchical solvers proposed in \cite{dimitrov2015efficient} (\texttt{lexls}) and in \cite{pfeiffer2023n} (\texttt{NIPM}). 

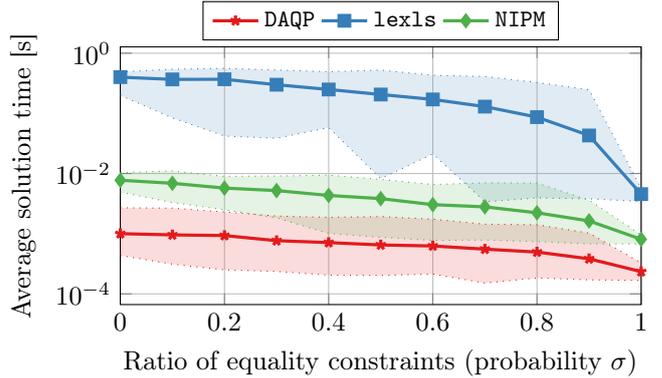
\begin{figure}
  \centering
  \begin{tikzpicture}[scale=1]
    \usepgfplotslibrary{fillbetween} 
    \pgfplotstableread{data/ratio.dat}{\ratio}
    \begin{axis}[
        xmin = 0, xmax=1,
        scale=1,
        ymode=log,
        %ymin=10^-5,ymax=2*10^-1,
        xlabel={Ratio of equality constraints (probability $\sigma$)},
        ylabel={Average solution time [s]},
        legend style={at ={(0.5,1.025)},anchor=south}, ymajorgrids,yminorgrids,xmajorgrids,
        x post scale=1,
        y post scale=0.6,
        legend cell align={left},legend columns=3,
        legend style={nodes={scale=0.9, transform shape}},
        ]
        \addplot [set19c1,very thick, mark=star] table [x={ratio}, y={daqpmean}] {\ratio}; 
        \addplot [set19c2,very thick, mark=square*] table [x={ratio}, y={lexlsmean}] {\ratio}; 
        \addplot [set19c3,very thick, mark=diamond*] table [x={ratio}, y={nipmmean}] {\ratio}; 

        \addplot [set19c1,dotted, name path=A] table [x={ratio}, y={daqpmax}] {\ratio}; 
        \addplot [set19c1,dotted, name path=B] table [x={ratio}, y={daqpmin}] {\ratio}; 
        \addplot [set19c1,fill opacity=0.15] fill between [of=A and B];

        \addplot [set19c2,dotted, name path=Al] table [x={ratio}, y={lexlsmax}] {\ratio}; 
        \addplot [set19c2,dotted, name path=Bl] table [x={ratio}, y={lexlsmin}] {\ratio}; 
        \addplot [set19c2,fill opacity=0.15] fill between [of=Al and Bl];

        \addplot [set19c3,dotted, name path=An] table [x={ratio}, y={nipmmax}] {\ratio}; 
        \addplot [set19c3,dotted, name path=Bn] table [x={ratio}, y={nipmmin}] {\ratio}; 
        \addplot [set19c3,fill opacity=0.15] fill between [of=An and Bn];
        \legend{\texttt{DAQP} , \texttt{lexls}, \texttt{NIPM}};
    \end{axis}
\end{tikzpicture}
  \caption{Average solution time for forming the prioritized intersection for randomly generated problems with a different ratio $\sigma$ of equality constraints. The dashed lines show the best/worst-case solution times.}
  \label{fig:ratio}
\end{figure}
We generate the sets for different values of the equality constraint probability $\sigma$, which varies the ratio of equality constraints and inequality constraints. For example, $\sigma=1$ means that all of the constraints are equality constraints, and $\sigma=0$ means that all of the constraints are inequality constraint. For each $\sigma$ we randomly generate 1000 prioritized intersections and solve them with Algorithm~\ref{alg:main}, \texttt{lexls} and \texttt{NIPM}. Figure~\ref{fig:ratio} shows the average and best/worst-case solution times for solving the sequence in \eqref{eq:hi-qp} from $i=1$ to $i=10$ (i.e., for computing the prioritized intersection.)

The results are shown in Figure \ref{fig:ratio}, which shows that \texttt{DAQP} outperforms both \texttt{lexls} and \texttt{NIPM}. Moreover, the speedup is greater for lower equality probability $\sigma$, supporting our claim that \texttt{DAQP} are better at handling inequality constraints compared with existing solver for handling prioritized constraints.

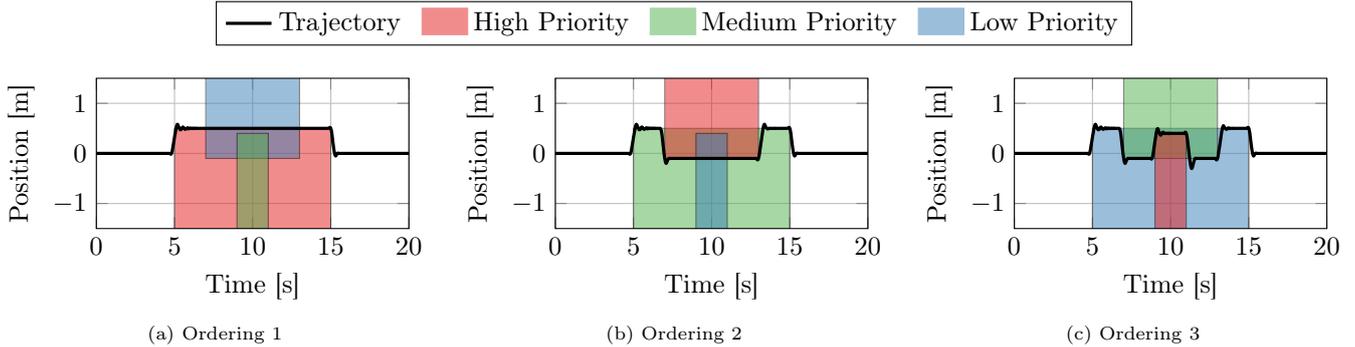
\begin{figure*}
    \centering
    \begin{tikzpicture}
        \pgfplotsset{filllegend/.style={
                legend image code/.code={
                    \fill[opacity=0.5] [#1] (0cm,-0.15cm) rectangle (0.6cm,0.2cm);
                },
        }}
        \begin{axis}[%
            hide axis,
            xmin=0, xmax=1,
            ymin=0, ymax=1,
            legend style={legend cell align=left, legend columns=4},
            legend style={/tikz/every even column/.append style={column sep=0.2cm}}
            ]
            \addlegendimage{black,very thick}
            \addlegendimage{filllegend=prio3}
            \addlegendimage{filllegend=prio4}
            \addlegendimage{filllegend=prio5}
            \legend{Trajectory, High Priority, Medium Priority, Low Priority}
        \end{axis}
    \end{tikzpicture}

    \vspace{10pt}
    \subfloat[Ordering 1]{%
        \makeplot{3}{5}{4}
    }
    \subfloat[Ordering 2]{%
        \makeplot{4}{3}{5}
    }
    \subfloat[Ordering 3]{%
        \makeplot{5}{4}{3}
    }
    \caption{Scenarios with differently prioritized time-dependent constraints on the lateral position.}
    \label{fig:scenarios}
\end{figure*}

\subsection{Autonomous driving}
\label{sec:mpc-exp}

To exemplify prioritized intersections in practice, we consider the control of the lateral position of a car moving at a constant speed $v$, which was considered in \cite{skibik2021feasibility}. The car is modeled with a bicycle model, where the states $x = [s\: \psi\: \beta\: \omega]$ consist of the lateral position $s$, the yaw angle $\psi$, the sideslip angle $\beta$, and the yaw rate $\omega$. The control is the steering angle $u = \delta_f$.

The continuous time dynamics is $\dot{x} = A x + B u$,
with the matrices
\begin{equation*}
  A \triangleq \begin{bmatrix}
      0 & v & v & 0 \\
      0 & 0 & 0 & 1 \\
      0 & 0 & -\frac{2 C_{\alpha}}{m v} & \frac{C_{\alpha} (\ell_r - \ell_f)}{m v^2} - 1 \\
          0 & 0 & \frac{C_{\alpha} (\ell_r - \ell_f)}{I_{zz}} & - \frac{C_{\alpha}(\ell_r^2 + \ell_f^2)}{I_{zz} v}
  \end{bmatrix},\quad 
  B \triangleq 
  \begin{bmatrix}
      0 \\ 0 \\ \frac{C_{\alpha}}{m v} \\ \frac{C_{\alpha} \ell_f}{I_{zz}}
  \end{bmatrix}.
\end{equation*}
The constants $C_{\alpha}$, $I_{zz}$, $\ell_r$, $\ell_f$, and $m$ characterize the car, and the particular  values for these in the experiments can be found in \cite{skibik2021feasibility}, or in the above-mentioned code. 

In additions, we have the state constraints
        $\left|\delta_f\right| \leq  \tfrac{\pi}{6}$,
        $\left|\delta_f -\beta - \tfrac{\ell_f}{v} \omega \right| \leq \tfrac{\pi}{22.5}$, and $\left|\tfrac{\ell_r}{v} \omega -\beta \right| \leq \tfrac{\pi}{22.5}$,
where the first constraint limits the steering angle, and the other two constraints limit the front and rear slip angles. The limits on the steering angle are due to mechanical limitations, while the limits on the slip angles are to avoid losing control of the car.
The road that the car is driving on has width $W$, which gives the constraint $|s| \leq \tfrac{W}{2}$.
Finally, we assume that there are three different obstacles in the lateral position of different importance, leading to the avoidance constraints 
    $s(t) \in [\underline{s}^i, \overline{s}^i] \text{ for } t \in [\underline{t}^i, \overline{t}^i]$
 for $i = 1,2,3$.

\begin{remark}[Obstacles] In practice it is common to have a extra layer that detects obstacles in the lateral position. Here we simply define constraints directly in lateral space since the focus is on highlighting how constraints can be prioritized, rather than the obstacle avoidance itself.
\end{remark}

These constraints are imposed according to the priorities
\begin{equation*}
    \begin{aligned}
        \left|\delta_f\right| &\leq  \tfrac{\pi}{6} & \quad(P1)\\
     \left|\delta_f -\beta - \tfrac{\ell_f}{v}\omega\right| &\leq \tfrac{\pi}{22.5},\qquad    \left|\tfrac{\ell_r}{v}\omega - \beta\right| \leq \tfrac{\pi}{22.5}  & \qquad (P2) \\
        |s| &\leq \tfrac{W}{2} & \qquad (P3) \\
        s(t) \in [\underline{s}^i, \overline{s}^i] &\text{ for } t \in [\underline{t}^i, \overline{t}^i] \text{ for } i =1,2,3. &(P4-P6)
    \end{aligned}
\end{equation*}
The reasoning for the ordering is as follows: the constraint on $\delta_f$ is mechanical and is impossible to override, hence it has the highest priority $(P1)$. The constraints on front and rear slip angles are necessary to retain the control of the car, it could be violated, but doing so would lead to unsafe behavior; hence, these constraints have the second highest priority $(P2)$. The constraint to stay on the road could also be violated if necessary, but should be avoided; hence the priority is moderate $(P3)$. Finally, the obstacles have different priorities, ranging from $(P4)$ to $(P6)$. These constraints should be avoided if possible, but can be violated if necessary.

The constraint hierarchy $(P1-P6)$ was incorporated in a linear MPC controller, which after state condensation is of the form~\eqref{eq:opt-ctrl} with $\mathcal{U}(x) = \pcap_{i=1}^6 \mathcal{U}_i(x)$, where $\mathcal{U}_i$ are polyhedra (see the experiment code for details) from the constraints in $(P1)-(P6)$. A horizon of $N=30$ time steps with a sample time of $10$ milliseconds was used in the MPC. More implementation details for the MPC example are given in the above-mentioned experiment code. Figure~\ref{fig:scenarios} shows the resulting lateral position for three different scenarios when the obstacles have different prioritizations. Since the constraints from the obstacles are conflicting, constraints corresponding to obstacles with lower priority are violated when necessary. This means that different prioritizations results in different trajectories, as illustrated in Figure~\ref{fig:scenarios}. The resulting solve times for computing the prioritized intersections in each scenario are shown in Figure~\ref{fig:auto-times}, which highlights that Algorithm~\ref{alg:main} is able to resolve conflicting constraint within the controller's sample time ($10$ milliseconds).
Note that the solution time is for both forming the prioritized intersection \emph{and} minimizing a conventional quadratic cost for steering the lateral position to $0$ subject to the resulting prioritized intersection.
Note that when more conflicts need to be resolved the solve times increase. Using existing solvers like \texttt{lexls} or \texttt{NIPM} would not fulfill the real-time requirements for this application, highlighting that the contribution of this paper allows for a wider application of prioritized constraint in real-time control applications. 

\begin{figure}
    \centering
    \begin{tikzpicture}[scale=1]
        \pgfplotstableread{data/scenario354.dat}{\dataone}
        \pgfplotstableread{data/scenario435.dat}{\datatwo}
        \pgfplotstableread{data/scenario543.dat}{\datathree}
        \begin{axis}[
            name = time,
            xmin=0,xmax=20,
            ymode=log,
            xlabel={Time [s]},
            ylabel={Solve time [s]},
            legend style={at ={(0.5,1.35)},anchor=north}, 
            legend cell align={left},legend columns=3,
            legend style={nodes={scale=0.9, transform shape}},
            ymajorgrids,yminorgrids,xmajorgrids,
            x post scale=1.0,
            y post scale=0.35,
            %legend cell align={left},legend columns=3,
            ]
            \addplot [set19c1,very thick] table [x={t}, y={tdaqp}] {\dataone}; 
            \addplot [set19c2,very thick] table [x={t}, y={tdaqp}] {\datatwo}; 
            \addplot [set19c3,very thick] table [x={t}, y={tdaqp}] {\datathree}; 
            \addplot[mark=none, black,dashed, very thick, samples=2,domain=0:20] {0.01};
            \legend{Ordering 1, Ordering 2, Ordering 3};
        \end{axis}
    \end{tikzpicture}
    \caption{Execution time to solve hierarchical problems of the form \eqref{eq:hi-qp} from $i=1,\dots, 8$ using Algorithm~\ref{alg:main} for different orderings of the constraints.}
    \label{fig:auto-times}
\end{figure}
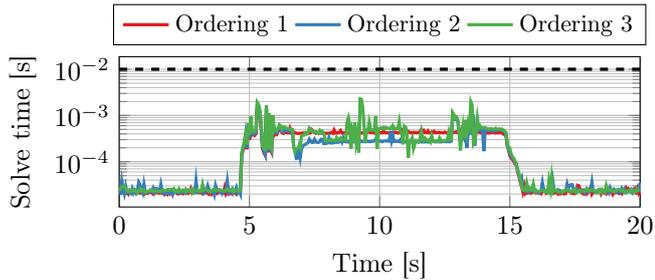

\section{Conclusion}
We have introduced a systematic way of handling prioritized constraints in optimization-based controllers. By introducing \emph{prioritized intersections}, we have given a formal way to resolve conflicting constraints, which unifies previous work.
We have also enabled the use of prioritized constraints in real-time application by proposing an efficient solver that computes prioritized intersections of polyhedra. This was validated in an MPC application for autonomous driving, where six different levels of conflicting constraint could be resolved in real-time. Finally, we showed that the proposed computational method outperforms state-of-the art hierarchical quadratic programming methods. 
The proposed framework for handling prioritized constraints is accessible in version 0.7.0 of \texttt{DAQP}\footnote{\url{https://github.com/darnstrom/daqp}}, and in version 0.6.0 of the MPC software package \texttt{LinearMPC.jl}\footnote{\url{https://github.com/darnstrom/LinearMPC.jl}}. 

\section*{Acknowledgements}
\noindent This work was partially supported by the Swedish Foundation for Strategic Research.

\bibliographystyle{IEEEtran}
\bibliography{lib.bib}

\begin{thebibliography}{10}
\providecommand{\url}[1]{#1}
\csname url@samestyle\endcsname
\providecommand{\newblock}{\relax}
\providecommand{\bibinfo}[2]{#2}
\providecommand{\BIBentrySTDinterwordspacing}{\spaceskip=0pt\relax}
\providecommand{\BIBentryALTinterwordstretchfactor}{4}
\providecommand{\BIBentryALTinterwordspacing}{\spaceskip=\fontdimen2\font plus
\BIBentryALTinterwordstretchfactor\fontdimen3\font minus \fontdimen4\font\relax}
\providecommand{\BIBforeignlanguage}[2]{{%
\expandafter\ifx\csname l@#1\endcsname\relax
\typeout{** WARNING: IEEEtran.bst: No hyphenation pattern has been}%
\typeout{** loaded for the language `#1'. Using the pattern for}%
\typeout{** the default language instead.}%
\else
\language=\csname l@#1\endcsname
\fi
#2}}
\providecommand{\BIBdecl}{\relax}
\BIBdecl

\bibitem{murphy2020beyond}
R.~R. Murphy and D.~D. Woods, ``Beyond {Asimov}: The three laws of responsible robotics,'' in \emph{Machine ethics and robot ethics}.\hskip 1em plus 0.5em minus 0.4em\relax Routledge, 2020, pp. 405--411.

\bibitem{rawlings2017model}
J.~B. Rawlings, D.~Q. Mayne, M.~Diehl \emph{et~al.}, \emph{Model predictive control: theory, computation, and design}.\hskip 1em plus 0.5em minus 0.4em\relax Nob Hill Publishing Madison, WI, 2017, vol.~2.

\bibitem{wabersich2023data}
K.~P. Wabersich, A.~J. Taylor, J.~J. Choi, K.~Sreenath, C.~J. Tomlin, A.~D. Ames, and M.~N. Zeilinger, ``Data-driven safety filters: Hamilton-jacobi reachability, control barrier functions, and predictive methods for uncertain systems,'' \emph{IEEE Control Systems Magazine}, vol.~43, no.~5, pp. 137--177, 2023.

\bibitem{garone2017reference}
E.~Garone, S.~Di~Cairano, and I.~Kolmanovsky, ``Reference and command governors for systems with constraints: A survey on theory and applications,'' \emph{Automatica}, vol.~75, pp. 306--328, 2017.

\bibitem{skibik2021feasibility}
T.~Skibik, D.~Liao-McPherson, T.~Cunis, I.~Kolmanovsky, and M.~M. Nicotra, ``A feasibility governor for enlarging the region of attraction of linear model predictive controllers,'' \emph{IEEE Transactions on Automatic Control}, vol.~67, no.~10, pp. 5501--5508, 2021.

\bibitem{johansen2013control}
T.~A. Johansen and T.~I. Fossen, ``Control allocation—a survey,'' \emph{Automatica}, vol.~49, no.~5, pp. 1087--1103, 2013.

\bibitem{scokaert1999feasibility}
P.~O. Scokaert and J.~B. Rawlings, ``Feasibility issues in linear model predictive control,'' \emph{AIChE Journal}, vol.~45, no.~8, pp. 1649--1659, 1999.

\bibitem{zeilinger2014soft}
M.~N. Zeilinger, M.~Morari, and C.~N. Jones, ``Soft constrained model predictive control with robust stability guarantees,'' \emph{IEEE Transactions on Automatic Control}, vol.~59, no.~5, pp. 1190--1202, 2014.

\bibitem{escande2014hierarchical}
A.~Escande, N.~Mansard, and P.-B. Wieber, ``Hierarchical quadratic programming: Fast online humanoid-robot motion generation,'' \emph{The International Journal of Robotics Research}, vol.~33, no.~7, pp. 1006--1028, 2014.

\bibitem{pfeiffer2023hierarchical}
K.~Pfeiffer, A.~Escande, P.~Gergondet, and A.~Kheddar, ``The hierarchical newton’s method for numerically stable prioritized dynamic control,'' \emph{IEEE Transactions on Control Systems Technology}, vol.~31, no.~4, pp. 1622--1635, 2023.

\bibitem{dimitrov2015efficient}
D.~Dimitrov, A.~Sherikov, and P.-B. Wieber, ``Efficient resolution of potentially conflicting linear constraints in robotics,'' 2015.

\bibitem{pfeiffer2023n}
K.~Pfeiffer, A.~Escande, and L.~Righetti, ``{$\mathcal{N}$IPM-HLSP}: an efficient interior-point method for hierarchical least-squares programs,'' \emph{Optimization and Engineering}, pp. 1--36, 2023.

\bibitem{censi2019liability}
A.~Censi, K.~Slutsky, T.~Wongpiromsarn, D.~Yershov, S.~Pendleton, J.~Fu, and E.~Frazzoli, ``Liability, ethics, and culture-aware behavior specification using rulebooks,'' in \emph{2019 International Conference on Robotics and Automation (ICRA)}.\hskip 1em plus 0.5em minus 0.4em\relax IEEE, 2019, pp. 8536--8542.

\bibitem{basso2020task}
E.~A. Basso and K.~Y. Pettersen, ``Task-priority control of redundant robotic systems using control lyapunov and control barrier function based quadratic programs,'' \emph{IFAC-PapersOnLine}, vol.~53, no.~2, pp. 9037--9044, 2020.

\bibitem{lee2023hierarchical}
J.~Lee, J.~Kim, and A.~D. Ames, ``Hierarchical relaxation of safety-critical controllers: Mitigating contradictory safety conditions with application to quadruped robots,'' in \emph{2023 IEEE/RSJ International Conference on Intelligent Robots and Systems (IROS)}.\hskip 1em plus 0.5em minus 0.4em\relax IEEE, 2023, pp. 2384--2391.

\bibitem{tyler1999propositional}
M.~L. Tyler and M.~Morari, ``Propositional logic in control and monitoring problems,'' \emph{Automatica}, vol.~35, no.~4, pp. 565--582, 1999.

\bibitem{kerrigan2002designing}
E.~C. Kerrigan and J.~M. Maciejowski, ``Designing model predictive controllers with prioritised constraints and objectives,'' in \emph{Proceedings. IEEE International Symposium on Computer Aided Control System Design}.\hskip 1em plus 0.5em minus 0.4em\relax IEEE, 2002, pp. 33--38.

\bibitem{vada2001linear}
J.~Vada, O.~Slupphaug, T.~A. Johansen, and B.~A. Foss, ``Linear {MPC} with optimal prioritized infeasibility handling: application, computational issues and stability,'' \emph{Automatica}, vol.~37, no.~11, pp. 1835--1843, 2001.

\bibitem{he2021lexicographic}
D.~He, H.~Li, and H.~Du, ``Lexicographic multi-objective {MPC} for constrained nonlinear systems with changing objective prioritization,'' \emph{Automatica}, vol. 125, p. 109433, 2021.

\bibitem{fulton2013intersection}
W.~Fulton, \emph{Intersection theory}.\hskip 1em plus 0.5em minus 0.4em\relax Springer Science \& Business Media, 2013, vol.~2.

\bibitem{kerrigan2000soft}
E.~C. Kerrigan and J.~M. Maciejowski, ``Soft constraints and exact penalty functions in model predictive control,'' in \emph{Control 2000 Conference, Cambridge}, 2000, pp. 2319--2327.

\bibitem{houska2024polyhedral}
B.~Houska, M.~A. M{\"u}ller, and M.~E. Villanueva, ``Polyhedral control design: Theory and methods,'' \emph{arXiv preprint arXiv:2412.13082}, 2024.

\bibitem{arnstrom2022daqp}
D.~Arnström, A.~Bemporad, and D.~Axehill, ``A dual active-set solver for embedded quadratic programming using recursive {LDL}$^{T}$ updates,'' \emph{IEEE Transactions on Automatic Control}, vol.~67, no.~8, pp. 4362--4369, 2022.

\bibitem{goldfarbidnani}
D.~Goldfarb and A.~Idnani, ``A numerically stable dual method for solving strictly convex quadratic programs,'' \emph{Mathematical Programming}, vol.~27, pp. 1--33, 9 1983.

\bibitem{bemporad2016qp}
A.~Bemporad, ``A quadratic programming algorithm based on nonnegative least squares with applications to embedded model predictive control,'' \emph{IEEE Transactions on Automatic Control}, vol.~61, no.~4, pp. 1111--1116, 2016.

\end{thebibliography}

%\appendix 
%\input{appendix.tex}
\end{document}